\documentclass[11pt,leqno]{article}
\usepackage{amsmath,amssymb,amsthm,latexsym}
\usepackage{indentfirst}
\usepackage{amsmath}

\newtheorem{theorem}{\bf \large Theorem}[section]
\newtheorem{PROPOSITION}{\bf \large Proposition}[section]
\newtheorem{corollary}{\bf \large Corollary}[section]
\newtheorem{lemma}{\bf \large Lemma}[section]
\newtheorem{ex}{\bf \large Example}[section]

\title{\textbf{Dupin Hypersurfaces in  Lorentzian Space forms}}
\author {{\bf Tongzhu Li$^{a,b}$, Changxiong Nie$^c$} \\
\small{$a$ Department of Mathematics, Beijing Institute of
Technology,} \\
\small{Beijing,China,100081, E-mail:litz@bit.edu.cn.} \\
\small{$b$ Beijing Key Laboratory on MCAACI,}
\small{Beijing,China,100081.} \\
\small{$c$ Faculty of Mathematics and Computer Sciences, Hubei
University, }\\
\small{Wuhan, China, 430062, E-mail:nie.hubu@yahoo.com.cn.}}
\date{}

\begin{document}
\maketitle
\begin{abstract}
Similar to the definition of Dupin hypersurface in Riemannian space forms, we define the spacelike Dupin hypersurface in Lorentzian space forms. As
conformal invariant objects, spacelike Dupin  hypersurfaces are studied in this paper
using the framework of conformal geometry. Further we classify the  spacelike Dupin hypersurfaces with constant M\"{o}bius curvatures,
which are the partition ratio of the principal curvatures of the  spacelike Dupin hypersurface.
\end{abstract}


\vskip 1 cm
\section{Introduction}
\par\medskip
Since Dupin surfaces were first studied by Dupin in 1822, the study of Dupin hypersurfaces in $\mathbb{R}^{n+1}$ has been a topic of increasing interest, (see \cite{cecil6,cecil8,cecil9,miya1,miya2,miya3,
pink1,pink2,pink3,tho1}), especially recently.
In this paper we study Dupin hypersurfaces in the Lorentzian space form $M^{n+1}_1(c)$.

Let $ \mathbb{R}^{n+2}_s$ be the real vector space $ \mathbb{R}^{n+2}$ with the
Lorentzian  product $\langle,\rangle_s$ given by
$$\langle X,Y\rangle_s=-\sum_{i=1}^sx_iy_i+\sum_{j=t+1}^{n+2}x_jy_j.$$
 For any $a>0$, the standard sphere $\mathbb{S}^{n+1}(a)$, the hyperbolic space $\mathbb{H}^{n+1}(-a)$,
the de sitter space $\mathbb{S}^{n+1}_1(a)$ and the anti-de sitter space $\mathbb{H}^{n+1}_1(-a)$ are defined by
\begin{equation*}
\begin{split}
\mathbb{S}^{n+1}(a)=\{x\in \mathbb{R}^{n+2}|x\cdot x=a^2\},~~\mathbb{H}^{n+1}(-a)=\{x\in \mathbb{R}^{n+2}_1|\langle x,x\rangle_1=-a^2\},\\
\mathbb{S}^{n+1}_1(a)=\{x\in \mathbb{R}^{n+2}_1|\langle x,x\rangle_1=a^2\},~~ \mathbb{H}^{n+1}_1(-a)=\{x\in
\mathbb{R}^{n+2}_2|\langle x,x\rangle_2=-a^2\}.
\end{split}
\end{equation*}
Let $M^{n+1}_1(c)$ be a Lorentz space form. When $c=0$,
$M^{n+1}_1(c)=\mathbb{R}^{n+1}_1$; When $c=1$, $M^{n+1}_1(c)=\mathbb{S}^{n+1}_1(1)$,
When $c=-1$, $M^{n+1}_1(c)=\mathbb{H}^{n+1}_1(-1)$.

Let $x:M^n\rightarrow M^{n+1}_1(c)$ be a spacelike hypersurface
in the Lorentzian space form $M^{n+1}_1(c)$. A curvature surface of $M^n$ is a smooth connected submanifold $S$
such that for each point $p\in S$, the tangent space $T_pS$ is equal to a principal space of the shape operator $\mathcal{A}$ of $M^n$ at $p$.
The hypersurface $M^n$ is called Dupin hypersurface if, along each curvature surface, the associated principal curvature is constant. The Dupin
hypersurface $M^n$ is called proper Dupin if the number $r$ of distinct principal curvatures is constant on $M^n$. The simple examples of spacelike Dupin hypersurface
are the  isoparametric hypersurfaces in $M^{n+1}_1(c)$, which are classified (see \cite{li,ma2,x}).

Similar to the Dupin hypersurfaces in Riemannian space forms, the Spacelike Dupin hypersurfaces in $M_1^{n+1}(c)$ are invariant
under the conformal transformations of $M_1^{n+1}(c)$.
Using Pinkall's method of constructed Dupin hypersurface in $\mathbb{R}^{n+1}$ (\cite{pink2}), we can  use the basic
constructions of building cylinders and  cones over a Dupin
hypersurface $W^{n-1}$ in $\mathbb{R}_1^n$ with $r-1$ principal curvatures to get a Dupin hypersurface $W^{n-1+k}$
in $\mathbb{R}_1^{n+k}$ with $r$ principal curvatures. Therefore we  show that, given any positive
integers $m_1,\cdots,m_r$ such that they sum to $n$, there exists a proper Dupin
hypersurface in $\mathbb{R}^{n+1}_1$ with $r$ distinct principal curvatures having respective
multiplicities $m_1,\cdots,m_r$. In general, these construction are local.

When the spacelike hypersurface $M^n$ has $r (\geq 3)$ distinct principal curvatures $\lambda_1,\cdots,\lambda_r$, the M\"{o}bius curvatures are defined by
$$\mathbb{M}_{ijs}=\frac{\lambda_i-\lambda_j}{\lambda_i-\lambda_s},~~~1\leq i,j,k\leq n.$$
The M\"{o}bius curvatures $\mathbb{M}_{ijs}$ are invariant under the conformal transformations  of $M_1^{n+1}(c)$ (see section 2). Our main results are
as follows,
\begin{theorem}\label{th1}
Let $x: M^n\to M^{n+1}_1(c)$ be a spacelike Dupin hypersurface in $M_1^{n+1}(c)$ with two distinct principal curvatures.  Then
locally $x$
is conformally equivalent to one of the following hypersurfaces,\\
(1), $\mathbb{S}^k(\sqrt{a^2+1})\times\mathbb{H}^{n-k}(-a)\subset\mathbb{S}^{n+1}_1, ~~a>0, ~1\leq k\leq n;$\\
(2), $\mathbb{H}^k(-a)\times\mathbb{H}^{n-k}(-\sqrt{1-a^2})\subset\mathbb{H}^{n+1}_1,~~0<a<1,~1\leq k\leq n;$\\
(3), $\mathbb{H}^k(a)\times\mathbb{R}^{n-k}
\subset\mathbb{R}^{n+1}_1,~~a>0,~0\leq k\leq n.$
\end{theorem}
\begin{theorem}\label{th2}
Let $x: M^n\to M^{n+1}_1(c)$ be a spacelike Dupin hypersurface in $M_1^{n+1}(c)$ with $r (\geq 3)$ distinct principal curvatures. If the  M\"{o}bius curvatures are constant, then
$r=3$, and locally $x$
is conformally equivalent to the following hypersurface,
$$
 x:\mathbb H^{q}(\sqrt{a^2-1})\times\mathbb S^{p}(a)\times
 \mathbb R^+\times\mathbb R^{n-p-q-1}\rightarrow
 \mathbb R^{n+1}_1 ,$$ defined by $$
  x(u',u'',t,u''')=(tu',tu'',u'''),$$ where $u'\in \mathbb H^{q}(\sqrt{a^2-1}), u''\in\mathbb S^{p}(a),u'''\in\mathbb R^{n-p-q-1},~~a>1.$
\end{theorem}
This paper is organized as follows. In section 2, we define some conformal invariants on a spacelike hypersurface and show that M\"{o}bius
curvatures are invariant under the conformal transformations  of $M_1^{n+1}(c)$. In section 3, we study the spacelike Dupin hypersurfaces in the framework of conformal geometry. In
section 4 and section 5, we give the proof of Theorem \ref{th1} and Theorem \ref{th2}, respectively.

\par\noindent
\section{Conformal geometry of Hypersurface in $M^{n+1}_1(c)$}
\par\medskip
In this section, following Wang's idea in paper \cite{w}, we define some conformal invariants on a spacelike  hypersurface
and give a congruent theorem of the spacelike hypersurfaces under the conformal
group of $M^{n+1}_1(c)$.

We denote by $C^{n+2}$ the cone in $\mathbb R^{n+3}_2$ and by
$\mathbb Q^{n+1}_1$ the conformal
compactification space in $\mathbb R P^{n+3}$,
$$C^{n+2}=\{X\in\mathbb R^{n+3}_2|\langle X,X\rangle_2=0,X\neq0\},$$
$$\mathbb Q^{n+1}_1=\{[X]\in\mathbb R P^{n+2}|\langle X,X\rangle_2=0\}.$$
Let $O(n+3,2)$ be the Lorentzian group of $\mathbb{R}^{n+3}_2$ keeping the
Lorentzian product $\langle X,Y\rangle_2$ invariant. Then $O(n+3,2)$ is a
transformation group on $\mathbb Q^{n+1}_1$ defined by
\begin{equation*}
T([X])=[XT], ~~~X\in C^{n+2}, ~~~T\in O(n+3,2).
\end{equation*}
Topologically $\mathbb Q ^{n+1}_1$ is identified with the compact
space $S^n\times S^1/S^0$, which is endowed by a standard Lorentzian
metric $h=g_{S^n}\oplus(-g_{S^1})$.  Then $\mathbb Q^{n+1}_1$ has
conformal metric $$[h]=\{e^\tau h|\tau\in C^ \infty(\mathbb
Q^{n+1}_1)\}$$ and $[O(n+3,2)]$ is the conformal transformation group of
$\mathbb Q^{n+1}_1$(see\cite{cahne,ma1,o}).

Denote $\pi=\{[X]\in \mathbb Q ^{n+1}_1|x_1=x_{n+2}\},~~\pi_-=\{[X]\in \mathbb Q ^{n+1}_1|x_{n+2}=0\},~~\pi_+=\{[X]\in \mathbb Q ^{n+1}_1|x_1=0\}$,
we can define the following conformal diffeomorphisms,
\begin{equation}\label{lor}
\begin{array}{l}
 \sigma_0 :
\mathbb{R}^{n+1}_1\rightarrow {\mathbb Q}^{n+1}_1\backslash\pi,~~~~\quad u\mapsto[( \frac{< u,
u>+1}{2},u, \frac{<u, u>-1}{2})],
   \\
\sigma_1: \mathbb{S}^{n+1}_1(1)\rightarrow  {\mathbb Q}^{n+1}_1\backslash\pi_+ ,~~\quad
 u\mapsto [(1,u)],  \\
\sigma_{-1}:  \mathbb{H}^{n+1}_1(-1)\rightarrow {\mathbb Q}^{n+1}_1\backslash\pi_-,~\quad
  u\mapsto [(u,1)]. \\
   \end{array}
\end{equation}
We may regard $\mathbb {Q}^{n+1}_1$ as the common compactification of $\mathbb{R}^{n+1}_1, \mathbb{S}^{n+1}_1(1), \mathbb{H}^{n+1}_1(-1)$.

Let $x:M^n\rightarrow M^{n+1}_1(c)$ be a spacelike hypersurface. Using $\sigma_c$, we obtain the hypersurface
in $\mathbb{Q}^{n+1}_1$, $\sigma_c\circ x:M^n\rightarrow \mathbb {Q}^{n+1}_1$.  From \cite{cahne}, we have the following theorem,
\begin{theorem}
Two hypersurfaces $x,\bar{x}:M^n\rightarrow M^{n+1}_1(c)$ are
conformally equivalent if and only if there exists $T\in O(n+3,2)$ such
that $\sigma_c\circ x=T(\sigma_c\circ \bar{x}):M^n\rightarrow \mathbb
Q^{n+1}_1$.
\end{theorem}
Since $x:M^n\rightarrow M^{n+1}_1(c)$ is a spacelike hypersurface, then $(\sigma_c\circ x)_*(TM^n)$ is a positive
definite subbundle of $T{\mathbb Q}^{n+1}_1$. For any local lift $Z$
of the standard projection $\pi: C^{n+2}\rightarrow
\mathbb{Q}^{n+1}_1$, we get a local lift $y=Z\circ\sigma_c\circ
x:U\rightarrow C^{n+1}$  of $\sigma_c\circ x:
M\rightarrow{\mathbb Q}^{n+1}_1$ in an open subset $U$ of
$M^n$. Thus $<dy,dy>=\lambda^2 dx\cdot dx$ is a local metric, which is
conformal to the induced metric $dx\cdot dx$. We denote by $\Delta$ and $\kappa$ the Laplacian
operator and the normalized scalar curvature with respect to the local positive
definite
metric $\langle\text dy, \text dy\rangle$, respectively.  Similar to Wang's proof of Theorem 1.2 in \cite{w},  we can get
the following theorem,
\begin{theorem}\label{t22}
Let $x:M^n\rightarrow M^{n+1}_1(c)$ be a spacelike hypersurface, then the 2-form $g=-(\langle\Delta y, \Delta
y\rangle-n^2\kappa)\langle\text dy, \text dy\rangle $ is a globally
defined conformal invariant.
Moreover, $g$ is positive definite at any non-umbilical point of
$M^n$.
\end{theorem}
We call $g$
the conformal metric of hypersurface $x$. There exists a unique lift
$$Y:M\rightarrow C^{n+2}$$ such that $g=<dY,dY>$.  We call $Y$ the conformal position
vector of $x$.
Theorem \ref{t22} implies that
\begin{theorem}
Two spacelike hypersurfaces $x,\bar{x}:M^n\rightarrow M^{n+1}_1(c)$
are conformally equivalent if and only if there exists $T\in O(n+3,2)$
such that $\bar{Y}=YT$, where $Y,\tilde Y$ are the conformal
position vector of $x,\bar{x}$, respectively.
\end{theorem}

Let $\{E_1, \cdots , E_n\}$ be a local orthonormal basis of $M^n$
with respect to $g$ with dual basis $\{\omega_1, \cdots ,
\omega_n\}$. Denote $Y_i=E_i(Y)$ and define
\begin{equation}\label{n1}
N=-\frac{1}{n}\Delta Y-\frac{1}{2n^2}\langle\Delta Y, \Delta
Y\rangle Y,
\end{equation}\label{innp}
where $\Delta$ is the Laplace operator of $g$, then we have
\begin{equation}
\langle N, Y\rangle=1,~ \langle N, N\rangle=0,~\langle N,
Y_k\rangle=0,~ \langle Y_i, Y_j\rangle=\delta_{ij},\quad1\leq i,j,k\leq n.
\end{equation}
We may decompose $\mathbb{R}^{n+3}_2$ such that
$$\mathbb{R}^{n+3}_2=\text{span}\{Y, N\}\oplus \text{span}
\{Y_1, \cdots , Y_n\}\oplus\mathbb V ,$$ where $\mathbb
V\bot\text{span}\{Y, N, Y_1, \cdots , Y_n\}$.  We call $\mathbb V$
the conformal normal bundle of $x$,
which is linear bundle. Let $\xi$ be a local section of $\mathbb{V}$
and $<\xi,\xi>=-1$, then $\{Y, N, Y_1, \cdots , Y_n, \xi\}$ forms a
moving frame in $\mathbb{R}^{n+3}_2$ along $M^n$. We write the structure
equations as follows,
\begin{equation}\label{struct}
\begin{split}
&\mathrm{d}Y=\sum_i\omega_iY_i,\\
&\mathrm{d}N=\sum_{ij}A_{ij}\omega_jY_i+\sum_iC_i\omega_i\xi,\\
&\mathrm{d}Y_i=-\sum_{ij}A_{ij}\omega_jY-\omega_iN+\sum_j\omega_{ij}Y_j+\sum_{ij}B_{ij}\omega_j\xi,\\
&\mathrm{d}\xi=\sum_iC_i\omega_iY+\sum_{ij}B_{ij}\omega_jY_i,
\end{split}
\end{equation}
 where $\omega_{ij}=-\omega_{ij}$ are the connection 1-forms on $M^n$ with respect to $\{\omega_1, \cdots ,
\omega_n\}$.  It is clear that
 $A=\sum_{ij}A_{ij}\omega_j\otimes\omega_i,~
B=\sum_{ij}B_{ij}\omega_j\otimes\omega_i,~ C=\sum_iC_i\omega_i$ are globally defined conformal
invariants. We call $A,~B$ and $C$ the conformal $2$-tensor, the conformal second fundamental form
 and the conformal $1$-form, respectively. The covariant derivatives
of these tensors with respect to  $g$ are defined by:
$$\sum_jC_{i, j}\omega_j=dC_i+\sum_kC_k\omega_{kj},$$
$$\sum_kA_{ij, k}\omega_k=dA_{ij}+\sum_kA_{ik}\omega_{kj}
+\sum_kA_{kj}\omega_{ki},$$
$$\sum_kB_{ij, k}\omega_k=dB_{ij}+\sum_kB_{ik}\omega_{kj}
+\sum_kB_{kj}\omega_{ki},$$
By
exterior differentiation of structure equations (\ref{struct}),
we can get the integrable conditions of the structure equations
$$A_{ij}=A_{ji},~~~B_{ij}=B_{ji},$$
\begin{equation}\label{stru1}
A_{ij, k}-A_{ik, j}=B_{ij}C_k-B_{ik}C_j,
\end{equation}
\begin{equation}\label{stru2}
B_{ij,k}-B_{ik,
j}=\delta_{ij}C_k- \delta_{ik}C_j,
\end{equation}
\begin{equation}\label{stru3}
C_{i, j}-C_{j, i}=\sum_k(B_{ik}A_{kj}-B_{jk}A_{ki}),
\end{equation}
\begin{equation}\label{stru4}
R_{ijkl}=B_{il}B_{jk}-B_{ik}B_{jl}+A_{ik}\delta_{jl}+A_{jl}\delta_{ik}-A_{il}\delta_{jk}-A_{jk}\delta_{il}.
\end{equation}
Furthermore, we have
\begin{equation}\label{cond1}
\begin{split}
&\text{tr}(A)=\frac{1}{2n}( n^2\kappa-1),\quad
R_{ij}=\text{tr}(A)\delta_{ij}+(n-2)A_{ij}+\sum_kB_{ik}
B_{kj},\\
&(1-n)C_i=\sum_jB_{ij,j},\quad \sum_{ij}B_{ij}^2=\frac{n-1}{n}, \quad
\sum_i B_{ii}=0,
\end{split}
\end{equation}
where $\kappa$ is the normalized scalar curvature of $g$. From
(\ref{cond1}), we see that when $n\geq3$, all coefficients in the
structure equations are determined by the conformal
metric $g$ and the conformal second fundamental form $B$, thus we get the following conformal congruent
theorem,
\begin{theorem}
Two spacelike hypersurfaces
$x, \bar{x}: M^n\rightarrow M^{n+1}_1(c) (n\geq3)$ are conformally
equivalent if and only if there exists a diffeomorphism $\varphi:
M^n\rightarrow M^n$ which preserves the conformal metric  and the
conformal second fundamental form.
\end{theorem}
Next we give the relations between the conformal invariants and
isometric invariants of $x:M^n\rightarrow M^{n+1}_1(c)$.

First we consider the spacelike hypersurface in  $\mathbb{R}^{n+1}_1$. Let $\{e_1, \cdots , e_n\}$ be an orthonormal
local basis for the induced metric $I=<dx,dx>$ with dual basis
$\{\theta_1, \cdots , \theta_n\}$. Let $e_{n+1}$ be a normal vector
field  of $x$ , and $<e_{n+1},e_{n+1}>=-1$. Then we have the first
and second fundamental forms $I, II$ and the mean curvature $H$, $I=\sum_i\theta_i\otimes\theta_i,~
II=\sum_{ij}h_{ij}\theta_i\otimes\theta_j,~ H=\frac{1}{n}\sum_{i}
h_{ii}$. Denote $\Delta_{M}$ the Laplacian and $\kappa_{M}$ the
normalized scalar curvature for $I$.  By structure equation and
Gauss equation of $x:M^n\rightarrow \mathbb{R}^{n+1}_1$ we get that
\begin{equation}\label{hh}
\Delta_{M}x=nHe_{n+1}, \quad\kappa_{M}=\frac{-1}{n(n-1)}(n^2|H|^2
-|II|^2).
\end{equation}
For $x:M^n\rightarrow \mathbb{R}^{n+1}_1$, there is a lift
$$y:M^n\rightarrow C^{n+2},\quad y=( \frac{<x, x>+1}{2},x, \frac{<x, x>-1}{2})
.$$
It follows from (\ref{hh}) that
$$\langle\Delta Y, \Delta Y\rangle-n^2\kappa=\frac{n}{n-1}(-|II|^2+n| H|^2)=-e^{2\tau}.$$ Therefore the conformal metric and conformal position vector of $x$
\begin{equation}\label{g}
\begin{split}
&g=\frac{n}{n-1}(|II|^2-n|H|^2)<\text dx,\text dx>:=e^{2\tau}I,\\
&Y=\sqrt{\frac{n}{n-1}(|II|^2-n|H|^2)}( \frac{<x, x>+1}{2},x,
\frac{<x, x>-1}{2}).
\end{split}
\end{equation}
Let $E_i=e^{-\tau}e_i$, then $\{E_i|1\leq i\leq n\}$ are the local
orthonormal basis for $g$, and with the dual basis
$\omega_i=e^\tau\theta_i$. Let $$y_i =(<x, e_i>,e_i,  <x,
e_i)),~~y_{n+1}=(<x, e_{n+1}>,e_{n+1},  <x, e_{n+1}>).$$ By some
calculations we can obtain that
\begin{equation}\label{coff0}
\begin{split}
Y=e^\tau y,\quad Y_i=e^\tau(\tau_iy+y_i),\quad \xi=-H y+y_{n+1},\\
-e^\tau N=\frac{1}{2}(|\nabla \tau|^2-|H|^2)
 y+\sum_i\tau_i
y_i+Hy_{n+1}+( 1,\vec{0}, 1),
\end{split}
\end{equation}
where $\tau_i=e_i(\tau)$ and $|\nabla \tau|^2=\sum_i\tau_i^2.$
By a direct calculation we get the following expression of the
conformal invariants $A,B, C$:
\begin{equation}\label{coff}
\begin{split}
&A_{ij}=e^{-2\tau}[\tau_i\tau_j- h_{ij}H
-\tau_{i,j}+\frac{1}{2}(-|\nabla\tau|^2+| H|^2) \delta_{ij}],\\
&B_{ij}=e^{-\tau}(h_{ij}-H \delta_{ij}), \\
&C_i=e^{-2\tau}(H\tau_i-H_i
-\sum_{j}h_{ij}\tau_j),
\end{split}
\end{equation}
where $\tau_{i,j}$ is the Hessian of $\tau$ for $I$ and $H_
i=e_i(H)$.

Using the same methods we can obtain relations between the conformal
invariants and isometric invariants of $x:M^n\rightarrow
\mathbb{S}^{n+1}_1(1)$ and $x:M^n\rightarrow \mathbb{H}^{n+1}_1(-1)$. We have the
following unitied expression of the conformal invariants $A,B, C$:
\begin{equation}\label{coff1}
\begin{split}
&A_{ij}=e^{-2\tau}[\tau_i\tau_j-\tau_{i,j}- h_{ij}H
+\frac{1}{2}(-|\nabla\tau|^2+| H|^2+\epsilon) \delta_{ij}],\\
&B_{ij}=e^{-\tau}(h_{ij}-H \delta_{ij}),\\
&C_i=e^{-2\tau}(H\tau_i-H_i
-\sum_{j}h_{ij}\tau_j),
\end{split}
\end{equation}
where $\epsilon=1$ for $x:M^n\rightarrow S^{n+1}_1(1)$, and
$\epsilon=-1$ for $x:M^n\rightarrow H^{n+1}_1(-1)$.

Let $\{b_1,\cdots,b_n\}$ be the eigenvalues of the conformal second fundamental form $B$, which are called conformal principal curvatures. Let
$\{\lambda_1,\cdots,\lambda_n\}$ be the principal curvatures.
From (\ref{coff}) and (\ref{coff1}), we have
\begin{equation}\label{bcon}
b_i=e^{-\tau}(\lambda_i-H), ~~i=1,\cdots,n.
\end{equation}
Clearly the number of distinct conformal principal curvatures is the same as
that of principal curvatures of $x$. Further, from equations (\ref{bcon}), the M\"{o}bius curvatures
\begin{equation}\label{moecur}
\mathbb{M}_{ijk}=\frac{\lambda_i-\lambda_j}{\lambda_i-\lambda_k}=\frac{b_i-b_j}{b_i-b_k},
\end{equation}
Combining equations (\ref{coff}),  (\ref{coff1}) and (\ref{moecur}), we have,
\begin{PROPOSITION}\label{pro4}
Let $x:M^{n}\rightarrow M_1^{n+1}(c)$ be a spacelike  hypersurface. Then the principal vectors and the conformal principal curvatures are
invariant under the conformal transformations of $M_1^{n+1}(c)$. In particular, the M\"{o}bius curvatures are invariant under the conformal transformations of $M_1^{n+1}(c)$.
\end{PROPOSITION}
It is then rather easily seen from (\ref{coff}) and (\ref{coff1}) that, if
all conformal principal curvatures $\{b_i\}$ are constant, then
M\"{o}bius curvatures $\mathbb{M}_{ijk}$ are constant for all $1\leq i, j, k \leq n$. Vice versa,
\begin{PROPOSITION}\label{pro4}
Let $x:M^{n}\rightarrow M_1^{n+1}(c)$ be a spacelike  hypersurface with $r (\geq 3)$ distinct principal curvatures. Then the M\"{o}bius curvatures
$\mathbb{M}_{ijk}$ are constant if and only if the conformal principal curvatures $\{b_1,\cdots,b_n\}$ are constant.
\end{PROPOSITION}
\begin{proof} It suffices to prove that the M\"{o}bius curvatures $\mathbb{M}_{ijk}$ are constant implies all conformal principal curvatures $b_i$ are constant.
First, for any tangent vector $X\in TM^n$, it is not hard to calculate that
$$\frac{X(b_i)-X(b_j)}{b_i-b_j}=\frac{X(b_i)-X(b_k)}{b_i-b_k}= \frac{X(b_j)-X(b_k)}{b_j-b_k} $$
from $\mathbb{M}_{ijk}$ being constant for all $1\leq i,j,k\leq n$.
Hence there exist $\mu$ and  $d$ such that
\begin{equation}\label{lb11}
X(b_j)=\mu b_j+d~~\text{for } j=1,\cdots,n.
\end{equation}
It is then immediate that \eqref{cond1} implies $d=0$ and $b_1X(b_1)+\cdots+b_nX(b_n)=0$, which implies $\mu=0$. Thus all $b_1,\cdots, b_n$ are constant.
\end{proof}

\par\noindent
\section{ Spacelike Dupin hypersurfaces in Lorentzian space forms}
Let $x:M^n\rightarrow M^{n+1}_1(c)$ be a spacelike Dupin hypersurface
in $M^{n+1}_1(c)$. For a principal curvature $\lambda$, we have principal space $\mathbb{D}_{\lambda}=\{X\in TM^n|\mathbb{A}X=\lambda X\}$. Then the spacelike hypersurface is Dupin if
and only if $X(\lambda)=0, X\in \mathbb{D}_{\lambda}$ for every principal curvature $\lambda$. The simple example of spacelike Dupin hypersurface
is the following isoparametric hypersurface in $M^{n+1}_1(c)$,
\begin{ex}\label{ex1}
 $\mathbb{H}^k(-a)\times\mathbb{R}^{n-k}
\subset\mathbb{R}^{n+1}_1,~~a>0,~~0\leq k\leq n.$
\end{ex}
\begin{ex}\label{ex2}
$\mathbb{S}^k(\sqrt{1+a^2})\times\mathbb{H}^{n-k}(-a)\subset\mathbb{S}^{n+1}_1,~~a>0,~~1\leq k\leq n.$
\end{ex}
\begin{ex}\label{ex3}
$\mathbb{H}^k(-a)\times\mathbb{H}^{n-k}(-\sqrt{1-a^2})\subset\mathbb{H}^{n+1}_1,~~0<a<1,~1\leq k\leq n.$
\end{ex}
In fact, these spacelike isoparametric  hypersurfaces are all spacelike isoparametric hypersurfaces in $M^{n+1}_1(c)$ (see \cite{li,ma2,x}). The following theorem
confirm that the spacelike Dupin hypersurface is conformally invariant.
\begin{theorem}
Let  $x:M^n\rightarrow M^{n+1}_1(c)$ be a spacelike Dupin hypersurface, and $\phi:M^{n+1}_1(c)\to M^{n+1}_1(c)$ a conformal transformation.
Then $\phi\circ x:M^n\rightarrow M^{n+1}_1(c)$ is a spacelike Dupin hypersurface.
\end{theorem}
\begin{proof}
Let $\{\lambda_1,\lambda_2,\cdots,\lambda_n\}$ denote its principal curvature, and
$\{e_1,e_2,\cdots,e_n\}$ be the orthonormal basis for $TM^n$ with the induced metric, consisting of unit principal vectors.
Therefore $\{E_1=e^{\tau}e_1, E_2=e^{\tau}e_2,\cdots,E_n=e^{\tau}e_n\}$ is the orthonormal basis for $TM^n$ with respect to the conformal metric $g$,
and $\{b_1=e^{-\tau}(\lambda_1-H),\cdots,b_n=e^{-\tau}(\lambda_n-H)\}$ are the conformal principal curvatures.
From (\ref{coff}) and (\ref{coff1}), we have
\begin{equation}\label{dupin}
\begin{split}
C_i&=e^{-\tau}(-e^{-\tau}H_i+\sum_j(h_{ij}-H\delta_{ij})(e^{-\tau})_j)\\
&=e^{-\tau}(-e^{-\tau}H_i+\sum_je_j((h_{ij}-H\delta_{ij})e^{-\tau})-e^{-\tau}\sum_je_j(h_{ij}-H\delta_{ij}))\\
&=e^{-\tau}(\sum_je_j(B_{ij})-\sum_je^{-\tau}He_j(h_{ij}))\\
&=E_i(b_i)-e^{-\tau}E_i(\lambda_i).
\end{split}
\end{equation}
Noting that the principal vectors are conformal invariants, Therefore $x$ is Dupin if and only if $C_i=E_i(b_i)$, which is invariant under the conformal transformation of $M^{n+1}_1(c)$.
\end{proof}
From equation (\ref{dupin}) and Proposition \ref{pro4}, the spacelike  Dupin hypersurfaces with constant M\"{o}bius curvatures can be completely characterized in terms of M\"{o}bius invariants,
namely,
\begin{theorem} \label{is1} Let $x:M^{n}\rightarrow M_1^{n+1}(c)$ be a spacelike Dupin hypersurface with $r (\geq 3)$ distinct principal curvatures.
Then the M\"{o}bius curvatures are constant if and only if the conformal 1-form vanishes and the conformal principal curvatures
are constant.
\end{theorem}
As a consequence of Proposition \ref{pro4}, one easily derives
\begin{corollary}\label{corm}
A spacelike Dupin hypersurface with constant M\"{o}bius curvatures is always proper.
\end{corollary}
Like as Pinkall's method in \cite{pink2}, we construct a new spacelike Dupin hypersurface
from a spacelike Dupin hypersurface.
\begin{PROPOSITION}
Let $u:M^k\to \mathbb{R}^{k+1}_1$ be an immersed  spacelike hypersurface. The cylinder over $u$ is defined as following
$$x:M^k\times \mathbb{R}^{n-k}\to \mathbb{R}^{k+1}_1\times \mathbb{R}^{n-k}=\mathbb{R}^{n+1}_1,~~~x(p,y)=(u(p),y).$$
If $u$ is a Dupin hypersurface, then cylinder $x$ is a spacelike Dupin hypersurface.
\end{PROPOSITION}
\begin{PROPOSITION}
Let $u:M^k\to \mathbb{S}^{k+1}_1$ be an immersed  spacelike hypersurface. The cone over $u$ is defined as following
$$x:M^k\times R^+\times \mathbb{R}^{n-k-1}\to \mathbb{R}^{n+1}_1,~~~x(p,t,y)=(tu(p),y).$$
If $u$ is a Dupin hypersurface, then cone $x$ is a spacelike Dupin hypersurface.
\end{PROPOSITION}
In general, these constructions introduce a new principal curvature of multiplicity $n-k$
which is constant along its curvature surface. The other principal curvatures are
determined by the principal curvatures of $M^k$, and the Dupin property is preserved
for these principal curvatures. Using these constructions we have the following result,
\begin{theorem}
Given positive integers $v_1,v_2,\cdots,v_r$¦Í with
$$v_1+v_2+\cdots+v_r=n.$$
there exists a proper spacelike Dupin hypersurface in $\mathbb{R}^{n+1}_1$ with $r$ distinct principal curvatures
having respective multiplicities $v_1,v_2,\cdots,v_r$.
\end{theorem}
Next we give the spacelike Dupin hypersurface which is not isoparametric in $M^{n+1}_1(c)$.
\begin{ex}\label{ex4}
Let $R^+$ be the half line of positive real numbers. For any two given natural numbers $p,q$ with $p+q<n$ and
a real number $a>1$, consider the hypersurface of warped product embedding
$$
 x:\mathbb H^{q}(\sqrt{a^2-1})\times\mathbb S^{p}(a)\times
 \mathbb R^+\times\mathbb R^{n-p-q-1}\rightarrow
 \mathbb R^{n+1}_1 ,$$ defined by $$
  x(u',u'',t,u''')=(tu',tu'',u'''),$$ where $u'\in \mathbb H^{q}(\sqrt{a^2-1}), u''\in\mathbb S^{p}(a),u'''\in\mathbb R^{n-p-q-1}.$
 \end{ex}
Let $b=\sqrt{a^2-1}$. One of the normal vector of $x$ can be taken as
$$e_{n+1}=(\frac{a}{b}u',\frac{b}{a}u'',0).$$
The first and second fundamental form of $x$ are given by
$$I=t^2(<du',du'>+du''\cdot du'')+dt\cdot dt+du'''\cdot du''',$$
$$II=-<dx,de_{n+1}>=-t(\frac{a}{b}<du',du'>+\frac{b}{a}du''\cdot du'').$$
Thus the mean curvature of $x$ $$H=\frac{-pb^2-qa^2}{nabt},$$
and $e^{2\tau}=\frac{n}{n-1}[\sum_{ij}h^2_{ij}-nH^2]=\frac{p(n-p)b^4-2pqa^2b^2+q(n-q)a^4}{(n-1)t^2}:=\frac{c^2}{t^2}.$

From (\ref{coff0}) and (\ref{coff1}), the conformal 1-form $C=0$, and the conformal metric and the conformal second fundamental form of $x$ are given by
\begin{equation}
\begin{split}
g&=c^2<du',du'>+c^2du''\cdot du''+\frac{c^2}{t^2}(dt\cdot dt+du'''\cdot du''')=g_1+ g_2+ g_3,\\
(B_{ij})&=(\underbrace{b_1,\cdots,b_1}_{q},\underbrace{b_2,\cdots,b_2}_{p},\underbrace{b_3,\cdots,b_3}_{n-p-q}),
\end{split}
\end{equation}
where $b_1=\frac{pb^2-(n-q)a^2}{nabc},~b_2=\frac{qa^2-(n-p)b^2}{nabc},~b_3=\frac{pb^2+qa^2}{nabc}.$

Therefore the embedding hypersurface $x$ is a spacelike Dupin hypersurface with three constant conformal principal curvatures.

Furthermore, the sectional curvatures of $(\mathbb H^{q}(\sqrt{a^2-1}), g_1)$, $(\mathbb S^{p}(a), g_2)$ and $(\mathbb R^+\times\mathbb R^{n-p-q-1}, g_3)$
are constant.

\par\noindent
\section{The proof of Theorem \ref{th1}}
To prove Theorem \ref{th1}, we need the following Lemma.
\begin{lemma}\label{abc1}
Let $x:M^n\rightarrow M^{n+1}_1(c)$ be a spacelike hypersurface
without umbilical points. If conformal invariants of $x$ satisfy
$C=0,$ and $ A=\mu B+\lambda g$
for some constant $\mu, \lambda$. Then $x$ is conformally equivalent to a spacelike hypersurface with
constant mean curvature and constant scalar curvature.
\end{lemma}
\begin{proof}
 Since $C=0$, From (\ref{stru3}), we can take the
local orthonormal basis $\{E_1,\cdots,E_n\}$ such that
\begin{equation}\label{ab}
(B_{ij})=diag(b_1,\cdots,b_n),~~(A_{ij})=diag(a_1,\cdots,a_n).
\end{equation}
Since $A=\mu B+\lambda g$, from structure equations (\ref{struct})
we get that
\begin{equation*}\label{c1}
dN-\lambda dY-\mu d\xi=0
\end{equation*}
and
$$d(N-\lambda Y-\mu \xi)=0.$$ Therefore we can find a constant
vector $e\in \mathbb{R}^{n+3}_2$ such that
\begin{equation}\label{e1}
N-\lambda Y-\mu \xi=e.
\end{equation}
From (\ref{innp}) and (\ref{e1}) we get that
\begin{equation*}
<e,e>=\mu^2-2\lambda,~~~ <Y,e>=1.
\end{equation*}
To prove the Theorem we consider the following three cases,\\
{\bf Case 1}\;\; $e$ is lightlike, i.e., $\mu^2-2\lambda=0$,\\
{\bf Case 2}\;\; $e$ is spacelike, i.e., $\mu^2-2\lambda>0$,\\
{\bf Case 3}\;\; $e$ is timelike, i.e., $\mu^2-2\lambda<0$.

First we consider {\bf Case 1,} $e$ is lightlike, i.e.,
$\mu^2-2\lambda=0$. Then there exists a $T\in O(n+3,2)$ such that
\begin{equation*}
\bar{e}=(-1,\vec{0},1)=eT=(N-\lambda Y-\mu \xi)T.
\end{equation*}
Let $\bar{x}:M^n\rightarrow \mathbb{R}^{n+1}_1$ be a hypersurface which its
conformal position vector is $\bar{Y}=YT$, then
$\bar{N}=NT,\bar{\xi}=\xi T$, and
\begin{equation}\label{const}
\bar{e}=\bar{N}-\lambda \bar{Y}-\mu \bar{\xi}, ~~<\bar{Y},\bar{e}>=1,~~ <\bar{\xi},\bar{e}>=-\mu.
\end{equation}
Writing $$\bar{Y}=e^{\bar{\tau}}( \frac{<\bar{x},
\bar{x}>+1}{2},\bar{x}, \frac{<\bar{x},
\bar{x}>-1}{2}), ~~~\bar{\xi}=-\bar{H}( \frac{<\bar{x},
\bar{x}>+1}{2},\bar{x}, \frac{<\bar{x},
\bar{x}>-1}{2})+\bar{y}_{n+1},$$ then from (\ref{coff}) and (\ref{const}),
we obtain that
\begin{equation*}
e^{\bar{\tau}}=1,~~ \bar{H}=\mu.
\end{equation*}
Since $\bar{Y}=( \frac{<\bar{x}, \bar{x}>+1}{2},\bar{x},
\frac{<\bar{x}, \bar{x}>-1}{2})$, then
$g=<d\bar{x},d\bar{x}>=\bar{I}.$ From (\ref{cond1}) we have
$$tr(A)=n\lambda=\frac{1}{2n}(n^2\kappa-1).$$ Since
$\kappa_M=\kappa$, so the mean curvature and scalar curvature of
hypersurface $\bar{x}$ are constant.

Next we consider {\bf Case 2,} $e$ is spacelike, i.e.,
$\mu^2-2\lambda>0$. Then there exists a $T\in O(n+3,2)$ such that
\begin{equation*}
\bar{e}=(\vec{0},\sqrt{\mu^2-2\lambda})=eT=(N-\lambda Y-\mu \xi)T.
\end{equation*}
Let $\bar{x}:M^n\rightarrow \mathbb{H}^{n+1}_1(-1)$ be a hypersurface which
its conformal position vector is $\bar{Y}=YT$, then
$\bar{N}=NT,\bar{\xi}=\xi T$, and
\begin{equation}\label{const1}
\bar{e}=\bar{N}-\lambda \bar{Y}-\mu \bar{\xi},~~<\bar{Y},\bar{e}>=1,~~ <\bar{\xi},\bar{e}>=-\mu.
\end{equation}
Writing
$\bar{Y}=e^{\bar{\tau}}(\bar{x},1),~~\bar{\xi}=-\bar{H}(\bar{x},
1)+\bar{y}_{n+1}$, then from (\ref{coff1}) and (\ref{const1}),
we obtain that
\begin{equation*}
e^{\bar{\tau}}=\frac{1}{\sqrt{\mu^2-2\lambda}},~~ \bar{H}=\mu.
\end{equation*}
Since $<d\bar{x},d\bar{x}>=(\mu^2-2\lambda)g$, so
$\kappa_M=\frac{1}{\mu^2-2\lambda}\kappa.$
Therefore the mean curvature and scalar curvature of hypersurface
$\bar{x}$ are constant.

Finally  we consider {\bf Case 3,} $e$ is timelike, i.e.,
$\mu^2-2\lambda<0$. Then there exists a $T\in O(n+3,2)$ such that
\begin{equation*}
\bar{e}=(-\sqrt{2\lambda-\mu^2},\vec{0})=eT=(N-\lambda Y-\mu \xi)T.
\end{equation*}
Let $\bar{x}:M^n\rightarrow \mathbb{S}^{n+1}_1(1)$ be a hypersurface which
its conformal position vector is $\bar{Y}=YT$, then
$\bar{N}=NT,\bar{\xi}=\xi T$, and
\begin{equation}\label{const2}
\bar{e}=\bar{N}-\lambda \bar{Y}-\mu \bar{\xi},~~<\bar{Y},\bar{e}>=1,~~ <\bar{\xi},\bar{e}>=-\mu.
\end{equation}
Writing
$\bar{Y}=e^{\bar{\tau}}(1,\bar{x}),~~\bar{\xi}=-\bar{H}(1,\bar{x})+\bar{y}_{n+1}$,
then from (\ref{coff1}) and (\ref{const2}),
 we obtain that
\begin{equation*}
e^{\bar{\tau}}=\frac{1}{\sqrt{2\lambda-\mu^2}}, ~~\bar{H}=\mu.
\end{equation*}
Since $<d\bar{x},d\bar{x}>=(2\lambda-\mu^2)g$, so
$\kappa_M=\frac{1}{2\lambda-\mu^2}\kappa.$
Therefore the mean curvature and scalar curvature of hypersurface
$\bar{x}$ are constant.
\end{proof}
Now we prove Theorem \ref{th1}. Let $x: M^n\to M_1^{n+1}(c)$ be a Dupin spacelike hypersurface with two distinct principal curvatures.
We take a local orthonormal basis $\{E_1,\cdots,E_n\}$ with respect to  $g$ such that under the basis
$$(B_{ij})=diag(\underbrace{b_1,\cdots,b_1}_k,\underbrace{b_2,\cdots,b_2}_{n-k}).$$
Using the equation (\ref{cond1}), we have
$$b_1=\frac{1}{n}\sqrt{\frac{(n-1)(n-k)}{k}},~~~b_2=\frac{-1}{n}\sqrt{\frac{(n-1)k}{n-k}}.$$
From (\ref{dupin}), we can obtain that
\begin{equation}\label{cc}
C=0.
\end{equation}
From equation (\ref{stru3}), we know that $[A,B]=0$. Thus we can take a local orthonormal basis $\{E_1,\cdots,E_n\}$ with respect to  $g$ such that under the basis
\begin{equation}\label{abec}
(B_{ij})=diag(\underbrace{b_1,\cdots,b_1}_k,\underbrace{b_2,\cdots,b_2}_{n-k}),~~(A_{ij})=diag(a_1,a_2,\cdots,a_n)).
\end{equation}
Since $b_1,b_2$ are constant, using the covariant derivatives of $(B_{ij})$, (\ref{stru2}) and (\ref{cc}) we can obtain
\begin{equation*}
B_{ij,l}=0,~~~1\leq i,j,l\leq n,~~~\omega_{i\alpha}=0,~~1\leq i\leq k,~k+1\leq\alpha\leq n,
\end{equation*}
which implies that
$$R_{i\alpha i\alpha}=0,~~~1\leq i\leq k,~k+1\leq\alpha\leq n.$$
Combining the equation (\ref{stru4}), we have
$$-b_1b_2+a_i+a_{\alpha}=0,~~~1\leq i\leq k,~k+1\leq\alpha\leq n,$$
thus $$a_1=\cdots=a_k,~~a_{k+1}=\cdots=a_n.$$
Using the covariant derivatives of $(A_{ij})$,
$$\sum_lA_{ij,l}\omega_l=dA_{ij}+\sum_lA_{il}\omega_{lj}
+\sum_lA_{lj}\omega_{li},$$ we can get
\begin{equation}
A_{ij,\alpha}=0,~~A_{\alpha\beta,i}=0,~~1\leq i,j\leq k,~~k+1\leq\alpha,\beta\leq n.
\end{equation}
Since $E_{\alpha}(a_1)=A_{ii,\alpha}=0,~~E_i(a_n)=A_{\alpha\alpha,i}=0$, combining $b_1b_2+a_i+a_{\alpha}=0$ we know that
$a_1=\cdots=a_k,~~a_{k+1}=\cdots=a_n$ are constant. Thus
$$(A_{ij})=diag(\underbrace{a_1,\cdots,a_1}_k,\underbrace{a_2,\cdots,a_2}_{n-k}).$$
Let $\mu=\frac{a_1-a_2}{b_1-b_2}$ and $\lambda=tr(A)=ka_1+(n-k)a_2$, then
\begin{equation*}
A=\mu B+\lambda g.
\end{equation*}
From Lemma \ref{abc1}, up to a conformal transformation, we know that
$e^{\tau}$ is constant. Combining (\ref{coff}), we know that the principal curvatures of  $x$ are constant.
From the classification of spacelike isoparametric hypersurfaces (see \cite{li,ma2,x}), up to a conformal transformation of $M_1^{n+1}(c)$, the Dupin hypersurface
$x$ is an isoparametric in $M_1^{n+1}(c)$. we finish the proof of Theorem \ref{th1}.

\par\noindent
\section{The proof of Theorem \ref{th2}}
To prove Theorem \ref{th2}, we need the following lemmas.
\begin{lemma}\label{cartanle}
Let $T=\sum_{ij}T_{ij}\omega_i\otimes\omega_j$ be a symmetric $(0,2)$ tensor with $r\geq 2$ distinct eigenvalues on $\mathbb{R}^n$,
and $F=\sum_{ijk}F_{ijk}\omega_i\otimes\omega_j\otimes\omega_k$ a symmetric $(0,3)$ tensor. Let $\{e_1,e_2,\cdots,e_n\}$
be the orthonormal basis, consisting of unit eigenvector of $T$. Under the basis, let
$$(T_{ij})=diag(\rho_1,\cdots,\rho_1,\rho_2,\cdots,\rho_2,\cdots,\rho_r,\cdots,\rho_r).$$
Then there does not exist the symmetric $(0,2)$ tensor $T$ satisfying $r\geq 3$ and
\begin{equation}\label{condt}
c-\rho_i\rho_j=\sum_k\frac{(F_{ijk})^2}{(\rho_i-\rho_k)(\rho_j-\rho_k)}, ~~\rho_i\neq\rho_j.
\end{equation}
\end{lemma}
\begin{proof}
We assume that there exists the symmetric $(0,2)$ tensor $T$ satisfying $r\geq 3$ and equation (\ref{condt}), we will find
a contradiction to prove the lemma.

We can assume that $\rho_1<\rho_2<\cdots<\rho_r$.  The equation (\ref{condt}) implies that
\begin{equation}\label{cartan0}
c-\rho_1\rho_2\geq0,~~c-\rho_2\rho_3\geq0,~~\cdots,~~c-\rho_k\rho_{k+1}\geq0,~~ c-\rho_{r-1}\rho_r\geq 0.
\end{equation}

For fixed induce $i$, the matrix $$\mathfrak{F}_{jk}:=\frac{(F_{ijk})^2}{(\rho_i-\rho_k)(\rho_j-\rho_k)(\rho_i-\rho_j)}$$
is antisymmetric for indices $j,k$, thus
\begin{equation}\label{cartan}
\sum_{j,\rho_j\neq\rho_i}\frac{c-\rho_i\rho_j}{\rho_i-\rho_j}=\sum_{j,k, \rho_j\neq\rho_i}\frac{(F_{ijk})^2}{(\rho_i-\rho_k)(\rho_j-\rho_k)(\rho_i-\rho_j)}=0.
\end{equation}

The proof of the lemma is divided into two cases: (1), $\rho_1<0$, (2), $\rho_1\geq0$.

For case (1), $\rho_1<0$. we have
$\rho_1\rho_2>\rho_1\rho_3>\cdots>\rho_1\rho_r.$ Combining (\ref{cartan0}), we have
$$c-\rho_1\rho_2\geq0,~~ c-\rho_1\rho_3>0,~\cdots,c-\rho_1\rho_r>0.$$ Thus
$$\frac{c-\rho_1\rho_j}{\rho_1-\rho_j}\leq0,~~\rho_j\neq \rho_1,$$
which is a contradiction with the equation (\ref{cartan}) for $i=1$.

For case (2), $\rho_1\geq0$. Then $\rho_r>\rho_{r-1}>\cdots>\rho_1\geq 0.$  Combining (\ref{cartan0}) we have
$c\geq\rho_r\rho_{r-1}>\rho_r\rho_{r-2}>\cdots>\rho_r\rho_1$, that is
$$c-\rho_r\rho_{r-1}\geq0, ~~c-\rho_r\rho_{r-1}>0,\cdots,c-\rho_r\rho_1>0.$$ Thus
$$\frac{c-\rho_r\rho_j}{\rho_r-\rho_j}\geq0,~~\rho_j\neq \rho_r,$$
which is a contradiction with the equation (\ref{cartan}) for $i=r$.
Thus we finish the proof of the lemma.
\end{proof}
Now let $M^n$ be a spacelike  Dupin hypersurface in $M_1^{n+1}(c)$ with $r (\geq 3)$ distinct principal curvatures. If the M\"{o}bius curvatures
are constant, then $C=0$, which implies $[A,B]=0$. Therefore we can choose a local orthonormal basis $\{E_1,\cdots,E_n\}$
with respect to the conformal metric $g$  such that
\begin{equation}\label{coe}
\begin{split}
&\left(A_{ij}\right)=diag(a_1,\cdots,a_n),\\
&\left(B_{ij}\right)=diag(b_1,\cdots,b_n)=diag(b_{\bar{1}},\cdots,b_{\bar{1}},b_{\bar{2}},\cdots,b_{\bar{2}},\cdots,b_{\bar{r}},\cdots,b_{\bar{r}}).
\end{split}
\end{equation}
For some $i$ fixed, in this section we define the index set $$[i]:=\{m|b_m=b_i\}.$$
Since the conformal principal curvatures $\{b_1,b_2,\cdots,b_n\}$ are constant, under the basis $\{E_1,\cdots,E_n\}$,  using the
covariant derivative of $B$, we have
\begin{equation}\label{codac}
(b_i-b_j)\omega_{ij}=\sum_kB_{ij,k}\omega_k.
\end{equation}
 We have the following results,
\begin{equation}\label{tensor1}
\begin{split}
&B_{ij,k}=0,~~when~~[i]=[j],~or~[j]=[k],~or~[i]=[k],\\
&\omega_{ij}=\sum_k\frac{B_{ij,k}}{b_i-b_j}\omega_k=\sum_{k\notin [i],[j]}\frac{B_{ij,k}}{b_i-b_j}\omega_k,~~when~~[i]\neq[j].
\end{split}
\end{equation}
Hence
\begin{equation}\label{w2}
\left\{\begin{split}
&B_{ij,k}=0 \text{ when }~[i]=[j] ~~or~~[i]=[k] \\
&\omega_{ij}=\sum_k\frac{B_{ij,k}}{b_i-b_j}\omega_k \text{ when }~~[i]\neq [j]
\end{split}\right.
\end{equation}
and
\begin{equation}\label{cur1}
R_{ijij}=\sum_{k\notin [i],[j]}\frac{2B^2_{ij,k}}{(b_i-b_k)(b_j-b_k)} \text{ when }~~[i]\neq [j].
\end{equation}
\begin{lemma}\label{blas}
Let $x:M^n\to M_1^{n+1}(c)$ be a spacelike  Dupin hypersurface with $r\geq 3$ distinct principal curvatures. If the  M\"{o}bius conformal curvatures are constant, then the eigenvalues of the conformal tensor
$\{a_1,\cdots,a_n\}$ are constant.
\end{lemma}
\begin{proof}
Since $A_{ij,k}=A_{ik,j}$, using the covariant derivative of $A$, we have
\begin{equation*}
(a_i-a_j)\omega_{ij}=\sum_kA_{ij,k}\omega_k,
\end{equation*}
which implies, from (\ref{w2}),
\begin{equation}\label{aw1}
(a_i-a_j)\frac{B_{ij,k}}{b_i-b_j}=A_{ij,k} \text{ when }~~[i]\neq [j].
\end{equation}
Hence we know
\begin{equation}\label{aw2}
E_i(a_j)=A_{jj,i}=A_{ij,j}=0 \text{ when }~~ [i]\neq [j]
\end{equation}
from $B_{ij,j}=0$. Now to verify that $a_j$ is a constant, we only need to prove
\begin{equation}\label{abw}
E_i(a_j)=0,~~i\in [j].
\end{equation}
For a fixed point $p\in M^n$ and $j\in\{1, \cdots, n\}$, it is either $B_{jk,l} = 0$ for all $1\leq k,l\leq n$ or $B_{jk,l}\neq 0$ for some $1\leq k,l\leq n$.
First assume it is the second case. In fact we may assume $B_{jk,l}\neq 0$ in a neighborhood of $p$ for some $j,k,l$ that have to be associated to
three distinct conformal principal curvatures. Therefore, from \eqref{aw1}, we obtain
\begin{equation*}
\frac{a_{j}-a_k}{b_{j}-b_k}=\frac{A_{jk,l}}{B_{jk,l}}=\frac{A_{lk,j}}{B_{lk,j}}=\frac{a_l-a_k}{b_l-b_k},
\end{equation*}
which implies
\begin{equation}\label{aw3}
a_{j}=(a_l-a_k)\frac{b_{j}-b_k}{b_l-b_k}+a_k.
\end{equation}
This easily implies (\ref{abw}). Next, suppose it is the first case. If there is a sequence of point $p_i\to p$ in $M^n$ such that the second cases happen on $p_i$ for
some $1\leq k,l\leq n$, then \eqref{abw} holds at $p$ due to the continuity. Otherwise, there is an open neighborhood $U\subset M^n$ of $p$ such that
$B_{jk,l} = 0$ for all $1\leq k,l\leq n$
in $U$. Therefore $R_{jkjk} = 0$ in $U$ from \eqref{cur1}. Hence, from \eqref{stru4}, we
derive
$$
a_j = b_jb_k - a_k \text{  in $U$ when $k\notin[j]$},
$$
which obviously implies \eqref{abw}. Thus the proof is complete.
\end{proof}
Since the eigenvalues of $A$ are constant, immediately we know
\begin{equation}\label{cotensor}
\begin{split}
&A_{ij,k}=0,~~when~~a_i=a_j,~or~a_j=a_k,\\
&\frac{a_i-a_j}{b_i-b_j}B_{ij,k}=A_{ij,k},~~when~~[i]\neq [j].
\end{split}
\end{equation}
Particularly the third equation in \eqref{cotensor} and $A_{ij,k}=A_{ik,j}$ implies
\begin{equation}\label{extra-vanish}
A_{ij,k} = 0 \text{ for $[j] = [i]$ and $k\notin [j]$},
\end{equation}
We define
$$V_{b_i} = Span\{E_m: m\in [i]\} \text{ or } V_{b_{\bar k}} = Span\{E_m: m\in [\bar k]\}.
$$
We can change the order of the subbasis in the eigenspace $V_{b_{\bar{k}}}$ such that
\begin{equation}\label{coe1}
\left(A_{ij}\right)|_{i,j\in [\bar{k}]}=diag(\underbrace{a_{k_1},\cdots,a_{k_1}},\underbrace{a_{k_2},\cdots,a_{k_2}}\cdots,\underbrace{a_{k_m},\cdots,a_{k_m}})
\end{equation}
for $a_{k_1}<a_{k_2}<\cdots<a_{k_m}.$
We then define the index sets
$$
(i):=\{l\in [i] | \quad a_l=a_i\} \text{ and }
(\bar{k_ i}) := \{l\in [\bar k]| \quad a_l = a_{k_i}\}.
$$
From \eqref{extra-vanish}, we have the following lemma,
\begin{lemma}\label{eigenspace}
Let $x:M^n\to M_1^{n+1}(c)$ be a spacelike  Dupin hypersurface with $r\geq 3$ distinct principal curvatures. If the M\"{o}bius curvatures are constant,
 Then, under the basis taken in (\ref{coe}) and (\ref{coe1}), for some $[\bar{k}]$ fixed,
$(i),(j)\in [\bar{k}]$ and $(i)\neq (j)$,
\begin{equation}\label{aa1}
(a_i-a_j)\omega_{ij}=\sum_{l\in[\bar{k}]}A_{ij,l}\omega_l
\end{equation}
and
\begin{equation}\label{aa2}
R_{ijij}=\sum_{l\in[\bar{k}],l\notin (i),(j)}\frac{2A_{ij,l}^2}{(a_i-a_l)(a_j-a_l)}.
\end{equation}
More importantly we have the generalized Cartan identity for $i\in[\bar{k}]$
\begin{equation}\label{aa4}
\sum_{j\in[\bar{k}],j\notin(i)}\frac{R_{ijij}}{a_i-a_j}=\sum_{j,l\in[\bar{k}],j,l\notin (i)}\frac{A_{ij,l}^2}{(a_i-a_l)(a_j-a_l)(a_i-a_j)}=0.
\end{equation}
\end{lemma}

\begin{lemma}\label{ble2} Let $x:M^n\to M_1^{n+1}(c)$ be a spacelike  Dupin hypersurface with $r\geq 3$ distinct principal curvatures. If  the M\"{o}bius  curvatures are constant, then $A|_{V_{b_{\bar{k}}}}$ has two distinct eigenvalues at most. Moreover
$$b_{\bar{k}}^2+a_{\bar{k}}+\bar{a}_{\bar{k}}=0$$
when $A|_{V_{b_{\bar{k}}}}$ has two distinct eigenvalues $a_{\bar{k}}$ and $\bar{a}_{\bar{k}}$.
\end{lemma}
\begin{proof}
For $a_{k_1}<a_{k_2}<\cdots<a_{k_m}$ and $i\in (k_1)$ and $j\in (k_2)$, it is easily seen from \eqref{aa2} that
$$
R_{ijij}=\sum_{l\in[\bar{k}],l\notin (\bar{k_1}),(\bar{k_2})}\frac{2A_{ij,l}^2}{(a_{k_1}-a_l)(a_{k_2}-a_l)}\geq 0.$$
Hence, from \eqref{stru4},
\begin{equation}\label{aa3}
R_{ijij}=-b_{\bar{k}}^2+a_i+a_j\geq -b_{\bar{k}}^2+a_{k_1}+a_{k_2} \geq 0, ~~i,j\in[\bar{k}] \text{ and } (i)\neq(j).
\end{equation}
Therefore, from the generalized Cartan identity \eqref{aa4} in Lemma \ref{eigenspace}, we get
\begin{equation}\label{aa5}
R_{ijij} = -b^2_{\bar k} + a_{k_1} + a_j =0, ~~i\in (\bar{k_1}) \text{ and } j\in[\bar{k}],j\notin(\bar{k_1}).
\end{equation}
The key of this proof is to realize that \eqref{aa5} allows us to further trim the generalized Cartan identity \eqref{aa4} for $i\in (\bar{k_2})$ into
\begin{equation}\label{aa6}
\sum_{j\in[\bar k], j\notin(\bar{k_1}), j\notin(\bar{k_2}))}\frac{R_{ijij}}{a_{k_2}-a_{j}} =0,
\end{equation}
which in turn implies
$$
R_{ijij} = -b^2_{\bar k} + a_{k_2} + a_j =0, ~~i\in (\bar{k_2}) \text{ and } j\in[\bar{k}],j\notin(\bar{k_2}).
$$
Thus, repeating the above argument,  we can get
\begin{equation}\label{aa7}
R_{ijij} = -b^2_{\bar k} + a_i + a_j =0 \text{ for all } i,j\in[\bar{k}] \text{ and } (i)\neq(j),
\end{equation}
which forces $m\leq 2$ and completes the proof.
\end{proof}
We may choose the orthonormal basis $\{E_1,\cdots,E_n\}$ such that
$\{E_1,\cdots,E_n\}$ such that
\begin{equation}\label{coe2}
\begin{split}
\left(B_{ij}\right)=diag&(\underbrace{b_{\bar{1}},\cdots,b_{\bar{1}}},\underbrace{b_{\bar{2}},\cdots,b_{\bar{2}}},
\cdots,\underbrace{b_{\bar{r}},\cdots,b_{\bar{r}}}),\\
\left(A_{ij}\right)=diag&(\underbrace{a_{\bar{1}},\cdots,a_{\bar{1}},\bar{a}_{\bar{1}},\cdots,\bar{a}_{\bar{1}}},\cdots,
\underbrace{a_{\bar{r}},\cdots,a_{\bar{r}},\bar{a}_{\bar{r}},\cdots,\bar{a}_{\bar{r}}}),
\end{split}
\end{equation}
where $a_{\bar{i}}$ and $\bar{a}_{\bar{i}}$ may be same and $b_{\bar{1}}<\cdots<b_{\bar{r}}$. We then define the following two index sets
$$
[i]=\{k\in\{1, 2, \cdots, n\}|\ b_k=b_i\}~~ \text{ and }~~ (i)=\{k\in [i]|\ a_k=a_i\}.$$
Let $s$ be the number of the distinct groups of indices in the collection $\{(1),(2),\cdots,(n)\}$ and label these distinct
groups of indices as $\{(\bar{1}),(\bar{2}),\cdots,(\bar{s})\}$. Clearly, we have  $(i)\subseteq [i] \text{ and } s\geq r.$
For any $i\in\{1, 2, \cdots, n\}$, we consider the pair $(a_i,b_i)$ and observe that
$$
(a_i,b_i)=(a_j,b_j) \text{ if and only if } (i) = (j).
$$
Hence one may write $(a_i,b_i)= (a_{(i)},b_{(i)})$ and there are exactly $s$ distinct pairs.  Let $W$ denote the set of all of the pairs, that is,
$$W=\{(a_{(\bar{1})},b_{(\bar{1})}),(a_{(\bar{2})},b_{(\bar{2})}),\cdots,(a_{(\bar{s})},b_{(\bar{s})})\}.$$
For a number $\varepsilon$ (including $\infty$) and a group $(i)$ fixed, we  define the set of pairs
$$S_{(i)}(\varepsilon):=\{(a_k,b_k)\in W| \ \frac{a_i-a_k}{b_i-b_k}=\varepsilon,~ k\notin (i)\} \bigcup\{(a_{(i)}, b_{(i)})\}.$$
From Lemma \ref{ble2} and the above definition of $S_{(i)}(\varepsilon)$, it is easy to verify the following properties:
\begin{lemma}\label{le2} Under the basis taken in (\ref{coe2}).
For a fixed index set $(i)$, the following hold: \\
(1) ${S}_{(i)}(\infty)$ can have at most two pairs;\\
(2) For two non-empty sets ${S}_{(i)}(\varepsilon_k),{S}_{(i)}(\varepsilon_l)$ and $\varepsilon_k\neq\varepsilon_l$,  ${S}_{(i)}(\varepsilon_k)\cap{S}_{(i)}(\varepsilon_l)=\{(a_{(i)}, b_{(i)})\}$;\\
(3) If  the set ${S}_{(i)}(\varepsilon)= \{(a_{(i)},b_{(i)}), (a_{(j)},b_{(j)})\}$ for $j\notin(i)$, then
\begin{equation}\label{abco1}
R_{klkl}=-b_ib_j+a_i+a_j=0 \text{ for all $k\in(i)$ and $l\in(j)$.}
\end{equation}
\end{lemma}
\begin{proof} These properties are all trivial except (3). It suffices to show that
$B_{kl,m}=0$ for all $m = 1, 2, \cdots, n$ when $k\in(i)$ and $l\in(j)$. The nontrivial cases are $k\in(i)\subset[i]$, $l\notin [i]$ and $m\notin [i]\cup[l]$.
Hence, from  the third equation in (\ref{cotensor}), we would have
$$
\frac{a_{m}-a_k}{b_{m}-b_k}=\frac{A_{km,l}}{B_{km,l}}=\frac{A_{lk,m}}{B_{lk,m}}=\frac{a_{l}-a_k}{b_{l}-b_k}
$$
if $B_{kl, m} = B_{km,l}$ were not vanishing.  That would imply $(a_m, b_m)\in S_{(i)}(\varepsilon)$ and a contradiction to assumption that $S_{(i)}(\varepsilon)$ has
only two pairs.  Thus the proof is complete.
\end{proof}
\begin{lemma} \label{le3}
 Under the basis taken in (\ref{coe2}). Then any set $$S_{(k)}(\varepsilon)=\{(a_{i_1},b_{i_1}),(a_{i_2},b_{i_2}),\cdots.
(a_{i_t},b_{i_t})\}$$ has only two pairs, that is $t=2$.
\end{lemma}
\begin{proof}
For $(a_i,b_i),(a_j,b_j)\in S_{(k_1)}(\varepsilon)$, we have
$\frac{a_i-a_j}{b_i-b_j}=\varepsilon$, thus there exist constant $d$ such that
$$a_i=\varepsilon b_i+d,~~(a_i,b_i)\in S_{(k_1)}(\varepsilon).$$
Let $\tilde{b}_i=b_i+\varepsilon$. From (\ref{stru4}) and (\ref{cur1}), we have
\begin{equation}\label{cartan1}
R_{ijij}=2\sum_k\frac{(B_{ij,k})^2}{(\tilde{b}_i-\tilde{b}_k)(\tilde{b}_j-\tilde{b}_k)}=2d+\varepsilon^2-\tilde{b}_i\tilde{b}_j=c-\tilde{b}_i\tilde{b}_j.
\end{equation}
The equation (\ref{cartan1}) implies that the tensor $B+\varepsilon g$ satisfying (\ref{condt}). If $t\geq 3$, from lemma \ref{cartanle}, we derive to
a contradiction. Thus $t=2$.
\end{proof}
Next we give the proof of Theorem \ref{th2}.
From lemma \ref{le2} and lemma \ref{le3}, we know that
\begin{equation}\label{rab}
R_{ijij}=0, ~b_i\neq b_j.
\end{equation}
From \eqref{cur1} we therefore observe that
$$B_{ij,k}=0,~~when~~i\in[{\bar{1}}],~~j\in[{\bar{2}}],~1\leq k\leq n.$$
We then consider $i\in [\bar 1]$ and $j\in [\bar 3]$ in equation (\ref{cur1}).
This time we notice that
$$(b_k-b_{\bar{1}})(b_k-b_{\bar{3}})>0,~~when~~k\notin [\bar{1}]\cup [\bar{2}] \cup[\bar 3]$$
and $B_{ij,k}=0,~i\in[{\bar{1}}],~k\in[{\bar{2}}]$. From \eqref{cur1} again we observe that
$$B_{ij,k}=0,~~when~~i\in[{\bar{1}}],~~j\in[\bar 2]\cup[{\bar{3}}],~1\leq k\leq n.$$
Repeatedly we can prove that $B_{ij,k}=0$ for $i\in [\bar 1]$ and $j\in [\bar 2]\cup[\bar 3]\cdots \cup[\bar r]$. Similarly we
can prove $B_{ij,k} = 0$ for all indices, thus $B$ is parallel.

{\bf Claim 1}: $r=3$.\\
We assume that $r>3$, we can take four distinct conformal principal curvatures $b_1,b_2,b_3,b_4$. Using (\ref{rab}) and (\ref{stru4}), we have
$$-b_1b_2+a_1+a_2=-b_1b_3+a_1+a+3=-b_2b_4+a_2+a_4=-b_3b_4+a_3+a_4=0,$$
which implies
$(b_1-b_4)(b_2-b_3)=0.$
This is a contradiction.

From (\ref{rab}), we have $a_i=a_j,~~[i]=[j]$ and
$$-b_1b_2+a_1+a_2=0,~~-b_1b_3+a_1+a_3=0,~~-b_2b_3+a_2+a_3=0.$$
Thus we can get
$$a_1=\frac{1}{2}(b_1b_2+b_1b_3-b_2b_3),~~a_2=\frac{1}{2}(b_1b_2+b_2b_3-b_1b_3),~~a_3=\frac{1}{2}(b_3b_2+b_1b_3-b_1b_2).$$
Since $B$ is parallel, using the definition of the covariant derivatives of $(B_{ij})$, we have
\begin{equation}\label{conn}
\omega_{ij}=0,~~b_i\neq b_j,
\end{equation}
which implies $$d\omega_i=\sum_{j\in [i]}\omega_{ij}\wedge\omega_j.$$
Therefore the eigenspaces $V_{b_1}$, $V_{b_2}$ and $V_{b_3}$ are integrable. Locally we can write $$M^n=M_1\times M_2\times M_3.$$
Let $[b_i]=\{k|b_k=b_i\}$, and
$$g_1=\sum_i\omega^2_i,~i\in [b_1],~~g_2=\sum_i\omega^2_i,~i\in [b_2],~~g_3=\sum_i\omega^2_i,~i\in [b_3].$$
Then we have
$$(M^n,g)=(M_1,g_1)\times(M_2,g_2)\times(M_3,g_3).$$
If $dim M_i\geq 2$, then $(M_i, g_i)$ is of constant curvature. Like as the proof in \cite{hul}, we know that $M^n$ is conformally equivalent to
the hypersurface given by example \ref{ex4}.

{\bf Acknowledgements:} Authors are supported by the
grant No. 11571037 and No. 11471021 of NSFC. And authors would like to thank Professor ChangPing Wang for his
encouragements and helps.

\end{document}